\documentclass[a4paper]{amsart}
\usepackage{amscd}
\usepackage{amsthm}
\usepackage{graphicx}
\usepackage{amsmath}
\usepackage{amsfonts}
\usepackage{amssymb}
\providecommand{\U}[1]{\protect\rule{.1in}{.1in}}
\newtheorem{theorem}{Theorem}

\newtheorem{corollary}[theorem]{Corollary}

\newtheorem{proposition}[theorem]{Proposition}
\newtheorem{remark}[theorem]{Remark}

\newcommand\supp{\mathop{\rm supp}}

\begin{document}

\title[Fractional series operators on discrete Hardy spaces]{Fractional series operators on discrete Hardy spaces}
\author{Pablo Rocha}
\address{Departamento de Matem\'atica, Universidad Nacional del Sur, Bah\'{\i}a Blanca, 8000 Buenos Aires, Argentina.}
\email{pablo.rocha@uns.edu.ar}
\thanks{\textbf{2.010 Math. Subject Classification}: 42B30, 42B25.}

\begin{abstract}
For $0 \leq \gamma < 1$ and a sequence $b=\{ b(i) \}_{i \in \mathbb{Z}}$ we consider the fractional operator $T_{\alpha, \beta}$ defined formally by 
$$(T_{\alpha, \beta} \, b)(j) = \sum_{i \neq \pm j} \frac{b(i)}{|i-j|^{\alpha} |i+j|^{\beta}} \,\,\,\,\,\, (j \in \mathbb{Z}),$$
where $\alpha, \beta > 0$ and $\alpha + \beta = 1 - \gamma$. The main aim of this note is to prove that the operator $T_{\alpha, \beta}$ is bounded from $H^{p}(\mathbb{Z})$ into $\ell^{q}(\mathbb{Z})$ for $0 < p < \frac{1}{\gamma}$ and $\frac{1}{q} = \frac{1}{p} - \gamma$. For 
$\alpha = \beta = \frac{1-\gamma}{2}$ we show that there exists $\epsilon \in \left(0, \frac{1}{3} \right)$ such that for every 
$0 \leq \gamma < \epsilon$ the operator $T_{\frac{1-\gamma}{2}, \frac{1-\gamma}{2}}$ is not bounded from $H^{p}(\mathbb{Z})$ into 
$H^{q}(\mathbb{Z})$ for $0 < p \leq \frac{1}{1 + \gamma}$ and $\frac{1}{q} = \frac{1}{p} - \gamma$.
\end{abstract}

\maketitle

\section{Introduction}

In this note, sequences considered are complex-valued unless otherwise explicitly stated. For a sequence 
$b= \{ b(i) \}_{i \in \mathbb{Z}}$, let
\[
\| b \|_{\ell^{p}(\mathbb{Z})} = \left\{\begin{array}{cc}
                  \left( \displaystyle{\sum_{i=-\infty}^{+\infty}} |b(i)|^{p} \right)^{1/p}, &  0 < p < \infty \\
                  \displaystyle{\sup_{i \in \mathbb{Z}}} \, |b(i)| \, , & \,\,\,\, p = \infty
                \end{array}. \right.
\]
A sequence $b$ is said to belong to $\ell^{p}(\mathbb{Z})$, $0 < p \leq \infty$, if $\| b \|_{\ell^{p}(\mathbb{Z})} < +\infty$.

At a Conference held in the summer of 1907, D. Hilbert announced the following inequality
\begin{equation} \label{Hilbert}
\sum_{i=1}^{\infty}\sum_{j=1}^{\infty} \frac{b(i) b(j)}{i+j} \leq 2\pi \sum_{i=1}^{\infty} [b(i)]^{2}
\end{equation}
for $b(i) \geq 0$ and $\sum_{i=1}^{\infty} [b(i)]^{2} < \infty$. Hilbert's proof was outlined by H. Weyl in his Inaugural-Dissertation 
(see \cite{Weyl}, pp. 83). Others proofs of (\ref{Hilbert}), essentially different from each other, have been published in \cite{Wiener},
\cite{Schur}, and \cite{Hardy1} (see also \cite{Hardy}, pp. 235-236).
The inequality in (\ref{Hilbert}) may be regarded as the starting point of all researches in the discrete setting.

Given a sequence $b = \{ b(i) \}_{i \in \mathbb{Z}}$ its Hilbert sequence $\widetilde{b}$ is defined by
\[
\widetilde{b}(j) = \frac{1}{\pi} \sum_{i \in \mathbb{Z} \setminus \{j\}} \frac{b(i)}{i-j}.
\]
This operator is known as \textit{discrete Hilbert transform}, such an operator was introduced by D. Hilbert in 1909. There are other operators with the same name. For instance, E. C. Titchmarsh in \cite{Titch} studied the behavior of the sequence
\begin{equation} \label{discrete Hilbert}
c(j) = \frac{1}{\pi} \sum_{i=-\infty}^{\infty} \frac{b(i)}{j+i+\frac{1}{2}}, \,\,\,\,\,\, (j \in \mathbb{Z}).
\end{equation}
\\
More precisely, he proved that if $b = \{ b(i) \}_{i \in \mathbb{Z}} \in \ell^{p}(\mathbb{Z})$, $1 < p < \infty$, then 
$c = \{ c(j) \}_{j \in \mathbb{Z}}$, given by (\ref{discrete Hilbert}), belongs to $\ell^{p}(\mathbb{Z})$ with 
$\| c \|_{\ell^{p}(\mathbb{Z})} \leq N_p \| b \|_{\ell^{p}(\mathbb{Z})}$, where $N_p$ is a number depending only on $p$. This result allowed him to obtain, by passing to the limit, that the Hilbert transform $\mathcal{H}$ defined by

\[
\mathcal{H}f(x) = p.v. \, \frac{1}{\pi}\int_{-\infty}^{\infty} \frac{f(t)}{t-x} dt, \,\,\,\,\,\,\, (x \in \mathbb{R}),
\] 
\\
is bounded on $L^{p}(\mathbb{R})$, $1 < p < \infty$ (see also \cite{Riesz}). The sequence in (\ref{discrete Hilbert}) is also known as the \textit{discrete Hilbert transform} of the sequence $b = \{ b(i) \}_{i \in \mathbb{Z}}$. For the sequel, we will consider the discrete Hilbert transform $H$ of a sequence $b = \{ b(i) \}_{i \in \mathbb{Z}}$ given by

\[
(Hb)(j) = \frac{1}{\pi} \sum_{i=-\infty}^{\infty} \frac{b(i)}{j+i+\frac{1}{2}},  \,\,\,\,\,\, (j \in \mathbb{Z}).
\]
\\
E. C. Titchmarsh also proved that $H(Hb) = b$ if $b \in \ell^{p}(\mathbb{Z})$, $1 < p <\infty$.

In \cite{HLP} (cf. also \cite{Hardy}, pp. 288), G. H. Hardy, J. E. Littlewood and G. P\'olya proved the following inequality
\begin{equation} \label{HLP ineq}
\left| \mathop{\sum \sum}_{i \neq j} |i-j|^{-\lambda} b(i) c(j) \right| \leq C \| b \|_{\ell^{p}(\mathbb{Z})} \| c \|_{\ell^{q}(\mathbb{Z})}
\end{equation}
for
\[
p > 1, \,\,\,\, q > 1, \,\,\,\, \frac{1}{p} + \frac{1}{q} > 1, \,\,\,\, \lambda = 2 - \frac{1}{p} - \frac{1}{q} \,\,\,\, 
(\text{so that} \,\, 0 < \lambda < 1).
\]
With this result H. G. Hardy and J. E. Littlewood in \cite{H-L} obtained the corresponding fractional integral theorem.

Given $0 < \gamma < 1$ and a sequence $b =\{ b(i) \}_{i \in \mathbb{Z}}$, we define the \textit{discrete Riesz potential} $I_{\gamma}$ by
\[
(I_{\gamma}b)(j) = \sum_{i \neq j} \frac{b(i)}{|i-j|^{1- \gamma}},  \,\,\,\,\,\, (j \in \mathbb{Z}).
\]
Taking $\lambda= 1 - \gamma$ in (\ref{HLP ineq}), it follows that the operator $I_{\gamma}$ is bounded from 
$\ell^{p}(\mathbb{Z})$ into $\ell^{q}(\mathbb{Z})$ for $1 < p < \gamma^{-1}$ and $\frac{1}{q} =  \frac{1}{p} - \gamma$.

The discrete Hardy space $H^{p}(\mathbb{Z})$, $0 < p < \infty$, consists of all sequences 
$b=\{ b(i) \}_{i \in \mathbb{Z}} \in \ell^{p}(\mathbb{Z})$ which satisfy $Hb \in \ell^{p}(\mathbb{Z})$. The "norm" of 
$b \in H^{p}(\mathbb{Z})$, $0 < p < \infty$, is defined as 

\begin{equation} \label{norm Hp}
\| b \|_{H^{p}(\mathbb{Z})} = \| b \|_{\ell^{p}(\mathbb{Z})} + \| Hb \|_{\ell^{p}(\mathbb{Z})}. 
\end{equation}

\

From (\ref{norm Hp}) and the boundedness on $\ell^{p}(\mathbb{Z})$, $1 < p < \infty$, of the discrete Hilbert transform $H$ it obtains that
$\| b \|_{\ell^{p}(\mathbb{Z})} \leq  \| b \|_{H^{p}(\mathbb{Z})}  \leq C_p \| b \|_{\ell^{p}(\mathbb{Z})}$, with $C_p$ independent of $b$, so
$H^{p}(\mathbb{Z}) = \ell^{p}(\mathbb{Z})$ when $1 < p < \infty$. For the range $0 < p \leq 1$, C. Eoff in \cite{Eoff} proved that 
$H^{p}(\mathbb{Z})$ is isomorphic to the Paley-Wiener space of entire functions $f$ of exponential type $\pi$ for which 
$\int_{\mathbb{R}} |f(x)|^{p} dx < +\infty$. 

By definition, it is clear that for $0 < p \leq 1$ the operator $H$ is bounded from $H^{p}(\mathbb{Z})$ into $\ell^{p}(\mathbb{Z})$, with
$\| Hb \|_{\ell^{p}(\mathbb{Z})} \leq \| b \|_{H^{p}(\mathbb{Z})}$ for all $b \in H^{p}(\mathbb{Z})$. Since 
$H^{p}(\mathbb{Z}) \subset \ell^{p}(\mathbb{Z}) \subset \ell^{2}(\mathbb{Z})$, for $0 < p \leq 1$, we have that $H(Hb) = b$ for every
$b \in H^{p}(\mathbb{Z})$. Thus, $H$ is a bounded operator on $H^{p}(\mathbb{Z})$, $0 < p \leq 1$, with
$\| Hb \|_{H^{p}(\mathbb{Z})} = \| b \|_{H^{p}(\mathbb{Z})}$ for all $b \in H^{p}(\mathbb{Z})$.

In \cite{Boza}, S. Boza and M. Carro proved equivalent definitions of the norms in $H^{p}(\mathbb{Z})$, $0 < p \leq 1$. One of their main goals is the discrete atomic decomposition of elements in $H^{p}(\mathbb{Z})$. By means of the atomic decomposition one can prove that the discrete Riesz potential $I_{\gamma}$ is bounded from $H^{p}(\mathbb{Z})$ into $\ell^{q}(\mathbb{Z})$ for $0 < p \leq 1$ and 
$\frac{1}{q}=\frac{1}{p} - \gamma$ (see Remark \ref{Igamma} below). 

Later, Y. Kanjin and M. Satake in \cite{Kanjin} obtained the molecular decomposition for members in $H^{p}(\mathbb{Z})$, $0 < p \leq 1$
(see also \cite{Komori}). As an application of the molecular decomposition they proved that the operator $I_{\gamma}$ is bounded from 
$H^{p}(\mathbb{Z})$ into $H^{q}(\mathbb{Z})$ for $0 < p \leq 1$ and $\frac{1}{q}=\frac{1}{p} - \gamma$.

Recently, Kwok-Pun Ho in \cite{Ho} generalized the famous discrete Hardy inequality to $0 < p \leq 1$, by using the atomic decomposition characterization of discrete Hardy spaces. For more results about discrete Hardy spaces see \cite{Liflyand} and references therein.

\

Let $0 \leq \gamma < 1$ and let $b = \{ b(i) \}_{i \in \mathbb{Z}}$ be a sequence. We define the fractional series operator 
$T_{\alpha, \beta}$ by
\begin{equation} \label{Tgamma}
(T_{\alpha, \beta} \, b)(j) = \sum_{i \neq \pm j} \frac{b(i)}{|i-j|^{\alpha} |i+j|^{\beta}}, \,\,\,\,\,\, (j \in \mathbb{Z}),
\end{equation}
\\
where $\alpha, \beta > 0$ and $\alpha + \beta = 1 - \gamma$. This operator is a discrete version of the following fractional type integral operator defined on $\mathbb{R}^{n}$

\begin{equation} \label{Talfa}
T_{\alpha}f(x) = \int_{\mathbb{R}^{n}} |x-A_1 y|^{-\alpha_1} \cdot \cdot \cdot |x- A _m y|^{-\alpha_m} f(y) \, dy,
\end{equation}
\\
where $0 \leq \alpha < n$, $m \in \mathbb{N} \cap \left(1 - \frac{\alpha}{n}, +\infty \right)$, the $\alpha_j$'s are positive constants such that $\alpha_1 + \cdot \cdot \cdot + \alpha_m = n - \alpha$ and the $A_j$'s are certain
$n \times n$ invertible matrices. The behavior of this kind of operators on classical and variable Hardy spaces was studied by the author and M. Urciuolo in \cite{Ro1}, \cite{Ro2} and \cite{Ro3}.

The germ of the operator in (\ref{Talfa}) appears in \cite{Ricci}, there F. Ricci and P. Sj\"ogren obtained the boundedness on 
$L^{p}(\mathbb{H}_1)$, $1 < p \leq +\infty$, for a family of maximal operators on the three dimensional Heisenberg group
$\mathbb{H}_1$. To get this result, they studied the $L^{2}(\mathbb{R})$ boundedness of the operator
\[
Tf(x) = \int_{\mathbb{R}} |x-y|^{\alpha-1}|(\beta-1)x- \beta y|^{-\alpha} f(y) \, dy,
\] 
for $\beta \neq 0, 1$ and $0 < \alpha < 1$.

The purpose of this note is to prove the $H^{p}(\mathbb{Z})$ - $\ell^{q}(\mathbb{Z})$ boundedness of the operator $T_{\alpha, \beta}$, given in (\ref{Tgamma}), for $0 < p < \gamma^{-1}$ and $\frac{1}{q} = \frac{1}{p} - \gamma$. We also prove that there exists 
$\epsilon \in \left(0, \frac{1}{3} \right)$ such that, for every $0 \leq \gamma < \epsilon$, the operator 
$T_{\frac{1-\gamma}{2}, \frac{1-\gamma}{2}}$ is not bounded from $H^{p}(\mathbb{Z})$ into $H^{q}(\mathbb{Z})$ for 
$0 < p \leq \frac{1}{1 + \gamma}$ and $\frac{1}{q} = \frac{1}{p} - \gamma$. This is an important difference between the discrete Hilbert transform $H$ and  $T_{\frac{1}{2}, \frac{1}{2}}$, and the discrete Riesz potential $I_{\gamma}$ and 
$T_{\frac{1-\gamma}{2}, \frac{1-\gamma}{2}}$ for $0 < \gamma < \epsilon$.

\

In Section 2, we introduce the discrete maximal and we recall the atomic decomposition of discrete Hardy spaces given in \cite{Boza}. We also state two corollaries, which are a consequence of the atomic decomposition. These two corollaries are necessary to make our 
counter-example.

In Section 3, we obtain the $H^{p}(\mathbb{Z})$ - $\ell^{q}(\mathbb{Z})$ boundedness of the operator $T_{\alpha, \beta}$, for 
$0 \leq \gamma <1$, $0 < p < \gamma^{-1}$ and $\frac{1}{q} = \frac{1}{p} - \gamma$.

In section 4, we give a counter-example which proves that $T_{\alpha, \beta}$ is not bounded from $H^{p}(\mathbb{Z})$ into 
$H^{q}(\mathbb{Z})$.

\

{\bf Notation.} Throughout this paper, $C$ will denote a positive real constant not necessarily the same at each occurrence. We set 
$\mathbb{N}_0 = \mathbb{N} \cup \{0\}$. With $\#A$ we denote the cardinality of a set $A \subset \mathbb{Z}$. Given a real number 
$s \geq 0$, we write $\lfloor s \rfloor$ for the integer part of $s$.

\section{Preliminaries}

Given a sequence $b = \{ b(i) \}_{i \in \mathbb{Z}}$ we define the centered maximal sequence $Mb$ by
\[
(Mb)(j) = \sup_{N \in \mathbb{N}_0} \frac{1}{2N+1} \sum_{|i-j|\leq N} |b(i)|.
\]
It is clear that $\| Mb \|_{\ell^{\infty}(\mathbb{Z})} \leq \|b \|_{\ell^{\infty}(\mathbb{Z})}$. Since 
$\ell^{p}(\mathbb{Z}) \subset \ell^{\infty}(\mathbb{Z})$ for every $0 < p < \infty$, it follows that if $b \in \ell^{p}(\mathbb{Z})$, 
$0 < p \leq \infty$, then $(Mb)(j) < +\infty$ for all $j \in \mathbb{Z}$.

The following result is a consequence of the harmonic analysis on spaces of homogeneous type applied to the space 
$(\mathbb{Z}, \mu, | \cdot|)$ where $\mu$ is the counting measure and $| \cdot |$ is the distance in $\mathbb{Z}$ (see \cite{Deng} or 
\cite{St}). We omit its proof.

\begin{theorem}\label{maximal} Let $b = \{ b(i) \}_{i \in \mathbb{Z}}$ be a sequence.
\begin{enumerate}
\item[(a)] If $b \in \ell^{1}(\mathbb{Z})$, then for every $\alpha > 0$
\[
\#\{ j : (Mb)(j) > \alpha \} \leq \frac{C}{\alpha} \| b \|_{\ell^{1}(\mathbb{Z})},
\]
where $C$ is a positive constant which does not depend on $\alpha$ and $b$.

\item[(b)] If $b \in \ell^{p}(\mathbb{Z})$, $1 < p \leq \infty$, then $Mb \in \ell^{p}(\mathbb{Z})$ and
\[
\| Mb \|_{\ell^{p}(\mathbb{Z})} \leq C_p \| b \|_{\ell^{p}(\mathbb{Z})},
\]
where $C_p$ depends only on $p$. 
\end{enumerate}
\end{theorem}

We observe that if $0 < p \leq \infty$ and $0 < q \leq 1$, then the mapping $b \to Mb$ is not bounded from $\ell^{p}(\mathbb{Z})$ into 
$\ell^{q}(\mathbb{Z})$. Indeed, taking $b = \{ b(i) \}$ such that $b(0) =1 $ and $b(i)=0$ for all $i \neq 0$ we have 
$(Mb)(j)= \frac{1}{2|j|+1}$ for all $j \in \mathbb{Z}$. So $b \in \ell^{p}(\mathbb{Z})$ but $Mb \notin \ell^{q}(\mathbb{Z})$ when 
$0 < p \leq \infty$ and $0 < q \leq 1$. We know that if $0 < p < q \leq \infty$, then $\ell^{p}(\mathbb{Z}) \subset \ell^{q}(\mathbb{Z})$ with $\| b \|_q \leq \| b \|_p$ for all $b \in \ell^{p}(\mathbb{Z})$. From this and Theorem \ref{maximal}-(b), we have that if $1 < q \leq \infty$ and $0 < p \leq q$, then $\| Mb \|_{\ell^{q}} \leq C_q \| b \|_{\ell^{p}}$. We also have that for $0 < q < p \leq \infty$ the inclusion $\ell^{q}(\mathbb{Z}) \subset \ell^{p}(\mathbb{Z})$ is strict and since $|b(j)| \leq (Mb)(j)$ for all $j$, it follows that the maximal operator $M$ is not bounded from $\ell^{p}(\mathbb{Z})$ into $\ell^{q}(\mathbb{Z})$ if $0 < q < p \leq \infty$. In the following proposition, we summarize all this.

\begin{proposition} Let $0 < p, \, q \leq \infty$. Then the maximal operator $M$ is bounded from $\ell^{p}(\mathbb{Z})$ into 
$\ell^{q}(\mathbb{Z})$ if and only if
\[
\left( 1/p, 1/q \right) \in \left\{ (x,y) \in [0,1] \times [0,1) : y \leq x \right\} \cup [1, \infty) \times [0,1).
\]
\end{proposition}

\

{\bf Discrete atoms.} Let $0 < p \leq 1 \leq q \leq \infty$, $p < q$ and $d$ be a non negative integer. A discrete $(p,q,d)$-atom centered at $n_0 \in \mathbb{Z}$ is a sequence $a= \{ a(i) \}_{i \in \mathbb{Z}}$ satisfying the conditions:
\\
\begin{enumerate}
\item $\supp (a) \subseteq \{n_0 - m, ..., n_0, ..., n_0 + m \}$, $m \geq 1$,
\\
\item $\|a\|_{\ell^{q}} \leq (2m+1)^{1/q - 1/p}$,
\\
\item $\displaystyle{\sum_{i=-\infty}^{+\infty}} i^{j} a(i) = 0$ for $j=0, 1, ..., d$.
\end{enumerate}
Here, $(1)$ means that the support of an atom is finite, $(2)$ is the size condition of the atom, and $(3)$ is called the \textit{cancellation moment condition}. Clearly, a $(p, \infty, d)$-atom is a $(p,q,d)$-atom, if $0 < p < q < \infty$. If $a$ is a $(p,q,d)$-atom,
then $\| a \|_{\ell^{p}} \leq 1$.

\

The atomic decomposition for members in $H^{p}(\mathbb{Z})$, $0 < p \leq 1$, developed in \cite{Boza} is as follows:

\begin{theorem} $($Theorem 3.13 in \cite{Boza}$)$ \label{decomp atomic} Let $0 < p \leq 1$, $d_p = \lfloor p^{-1} - 1 \rfloor$ and 
$b \in H^{p}(\mathbb{Z})$. Then there exist a sequence of $(p, \infty, d_p)$-atoms $\{ a_k \}_{k=0}^{+\infty}$, a sequence of scalars $\{ \lambda_k \}_{k=0}^{+\infty}$ and a positive constant $C$, which depends only on $p$, with $\sum_{k=0}^{+\infty} |\lambda_k |^{p} \leq C \| b \|_{H^{p}(\mathbb{Z})}^{p}$ such that $b = \sum_{k=0}^{+\infty} \lambda_k a_k$, where the series converges in $H^{p}(\mathbb{Z})$.
\end{theorem}

\begin{corollary} \label{moment cond} If $\{ b(i) \}_{i=-\infty}^{+\infty} \in H^{p}(\mathbb{Z})$, $0 < p \leq 1$, then $\sum_{i=-\infty}^{+\infty} b(i) =0$.
\end{corollary}

\begin{proof} By Theorem \ref{decomp atomic}, given $b \in H^{p}(\mathbb{Z})$, $0 < p \leq 1$, we can write $b = \sum_{k=0}^{+\infty} \lambda_k a_k$ where the $a_k$'s are $(p, \infty, d_p)$-atoms, and the series converges in $H^{p}(\mathbb{Z})$. Since $H^{p}(\mathbb{Z}) \subset \ell^{p}(\mathbb{Z}) \subset \ell^{1}(\mathbb{Z})$ embed continuously, the series also converges in $\ell^{1}(\mathbb{Z})$. Then for each $N$ fixed, by the cancellation moment condition of the atoms $a_k$, we have
\begin{equation} \label{ineq}
\left| \sum_{i=-\infty}^{+\infty} b(i) \right| = \left| \sum_{i=-\infty}^{+\infty} \left( b(i) - \sum_{k=0}^{N} \lambda_k a_k(i) \right)  \right| \leq \sum_{i=-\infty}^{+\infty} \left| b(i) - \sum_{k=0}^{N} \lambda_k a_k(i) \right|. 
\end{equation}
Finally, letting $N \to +\infty$ on the right-hand side of (\ref{ineq}), we obtain $\sum_{i=-\infty}^{+\infty} b(i) =0$.
\end{proof}

Given an integer $L \geq 0$, we define the set of sequences $\mathcal{D}_{L}$ by
\[
\mathcal{D}_{L} = \left\{ c = \{c(i)\}_{i \in \mathbb{Z}} : \# \supp(c) < +\infty \, \text{and} 
\sum_{i=-\infty}^{+\infty} i^{j} c(i)=0, \, \text{for} \,\, j= 0, 1, ..., L \right\}.
\]

\begin{corollary} \label{dense} If $0 < p \leq 1$ and $L$ is an arbitrary integer such that $L \geq \lfloor p^{-1} - 1 \rfloor$, then the set $\mathcal{D}_{L}$ is dense in $H^{r}(\mathbb{Z})$ for each $p \leq r \leq 1$.
\end{corollary}

\begin{proof} By checking the proof of Lemma 3.12 of Boza and Carro in \cite{Boza}, we see that in the atomic decomposition of an arbitrary element in 
$H^{r}(\mathbb{Z})$ one can always choose atoms with additional vanishing moments. This is, if $L$ is any fixed integer with $L \geq d_{r}$ and 
$b \in H^{r}(\mathbb{Z})$, then there exists an atomic decomposition for $b$ such that all moments up to order $L$ of the atoms are zero. For 
$0 < p \leq r \leq 1$ we have that $d_p \geq d_r$. If $L$ is any fixed integer such that $L \geq d_p$, then 
$\textit{span} \{ (r, \infty, L) - \text{atoms} \} \subset \mathcal{D}_{L} \subset H^{p}(\mathbb{Z}) \subset H^{r}(\mathbb{Z})$ for 
$p \leq r \leq 1$, so the corollary follows.
\end{proof}

\begin{proposition}\label{partialder} Let $0 \leq \gamma < 1$ and $\alpha, \beta >0$ such that $\alpha + \beta = 1 - \gamma$, and let $K : \mathbb{R}^{2} \setminus \{ y = \pm x \} \to \mathbb{R}$ be the function given by
\[
K(x,y) = |x-y|^{-\alpha} |x+y|^{-\beta}.
\]
Then
\[
\left| \frac{\partial^{N}}{\partial x^{N}} K(x,y) \right| + \left| \frac{\partial^{N}}{\partial y^{N}} K(x,y) \right| \leq C \, |x-y|^{-\alpha} |x+y|^{-\beta} \left( |x-y|^{-1} + |x+y|^{-1} \right)^{N},
\]
for every $N \in \mathbb{N}$, where $C$ is a positive constant independent of $x, y$.
\end{proposition}

\begin{proof} This result is a particular case of Lemma 1 in \cite{Ro2}.
\end{proof}

\section{Main Results}

Let $0 \leq \gamma < 1$, $0 < p < \gamma^{-1}$ and $\frac{1}{q}=\frac{1}{p} - \gamma$. In this section we study the 
$H^{p}(\mathbb{Z})$ - $\ell^{q}(\mathbb{Z})$ boundedness of the operator $T_{\alpha, \beta}$ defined by
\begin{equation} \label{Tgama}
(T_{\alpha, \beta} \, b)(j) = \sum_{i \neq \pm j} \frac{b(i)}{|i-j|^{\alpha} |i+j|^{\beta}} \,\,\,\,\,\, (j \in \mathbb{Z}),
\end{equation}
where $\alpha, \beta > 0$ and $\alpha + \beta = 1 - \gamma$.

\begin{theorem} \label{lpq} Let $0 \leq \gamma < 1$. If $T_{\alpha, \beta}$ is the operator given by $(\ref{Tgama})$, $1 < p < \gamma^{-1}$ and $\frac{1}{q} = \frac{1}{p} - \gamma$, then
\[
\| T_{\alpha, \beta} \, b \|_{\ell^{q}(\mathbb{Z})} \leq C \| b \|_{\ell^{p}(\mathbb{Z})},
\]
where $C$ depends only on $\alpha$, $\beta$, $p$ and $q$.
\end{theorem}

\begin{proof} Given a sequence $b = \{ b(i) \}_{i \in \mathbb{Z}}$ we put $|b| = \{ |b(i)| \}_{i \in \mathbb{Z}}$. We study the cases 
$0 < \gamma < 1$ and $\gamma = 0$ separately. For $0 < \gamma < 1$ it is easy to check that

\[
|(T_{\alpha, \beta} \, b)(j)| \leq (I_{\gamma}|b|)(j) + (I_{\gamma}|b|)(-j), \,\,\,\, \forall j \in \mathbb{Z}.
\]
\\
So, the $\ell^{p}(\mathbb{Z})$ - $\ell^{q}(\mathbb{Z})$ boundedness of $T_{\alpha, \beta}$ ($0 < \gamma < 1$) follows from the boundedness of the discrete Riesz potential $I_{\gamma}$.

For the case $\gamma = 0$, we introduce the auxiliary operator $\widetilde{T}_{\alpha, \beta}$ defined by 
$(\widetilde{T}_{\alpha, \beta} \, b)(j) = (T_{\alpha, \beta} \, b)(j)$ if $j \neq 0$ and $(\widetilde{T}_{\alpha, \beta} \, b)(0)=0$. Since 
\[
|(T_{\alpha, \beta} \, b)(0)| \leq 2^{1/p'} \| \{ i^{-1} \} \|_{\ell^{p'}(\mathbb{N})} \| b \|_{\ell^{p}(\mathbb{Z})} < +\infty, \, \forall \, 1 \leq p < \infty,
\]
it suffices to show that $\widetilde{T}_{\alpha, \beta}$ is bounded on $\ell^{p}(\mathbb{Z})$, $1 < p < +\infty$. We will show that the operator $\widetilde{T}_{\alpha, \beta}$ is bounded from $\ell^{p}(\mathbb{Z})$ into $\ell^{p, \infty}(\mathbb{Z})$ for each 
$1 \leq p <  \infty$. Then the $\ell^{p}(\mathbb{Z})$ boundedness of $\widetilde{T}_{\alpha, \beta}$ will follow from the Marcinkiewicz interpolation theorem (see Theorem 1.3.2 in \cite{Grafakos}) and, with it, that of $T_{\alpha, \beta}$.

Given $j_0 \neq 0$ fixed, we write $\mathbb{Z} \setminus \{ -j_0, j_0 \} = I_1 \cup I_2 \cup I_3$ where
$$I_1 = \{ i \in \mathbb{Z} : 0 < |i - j_0 | \leq |j_0| \}, \,\,\,\,\,\, I_2 = \{ i \in \mathbb{Z} : 0 < |i + j_0 | \leq |j_0| \},
\,\, \text{and} \,\,\,\,\,\, I_3 = \{ i \in \mathbb{Z} : |i| > 2|j_0| \}.$$
Then
\[
|(\widetilde{T}_{\alpha, \beta} \, b)(j_0)|=|(T_{\alpha, \beta} \, b)(j_0)| \leq \left( \sum_{i \in I_1} + \sum_{i \in I_2} + 
\sum_{i \in I_3} \right) \frac{|b(i)|}{|i-j_0|^{\alpha} |i+j_0|^{\beta}}.
\]
First, we estimate the sum on $I_1$. If $i \in I_1$, then $|i + j_0| = |2j_0 +i - j_0| \geq |j_0|$. So
\[
\sum_{i \in I_1} \frac{|b(i)|}{|i-j_0|^{\alpha} |i+j_0|^{\beta}} \leq \frac{1}{|j_0|^{\beta}} \sum_{0 < |i-j_0| \leq |j_0|} 
\frac{|b(i)|}{|i-j_0|^{\alpha}} =:S_1.
\]
Now, we take $k_0 \in \mathbb{N}_0$ such that $2^{k_0} \leq |j_0| < 2^{k_0 + 1}$, thus
\[
S_1 = \sum_{k=0}^{k_0} \frac{1}{|j_0|^{\beta}} \sum_{2^{-(k+1)}|j_0| < |i-j_0| \leq 2^{-k}|j_0|} \frac{|b(i)|}{|i-j_0|^{\alpha}}
\leq \sum_{k=0}^{k_0} \frac{2^{(k+1)\alpha}}{|j_0|} \sum_{|i-j_0| \leq \lfloor 2^{-k}|j_0| \rfloor} |b(i)|
\]
\[
= 2^{1+\alpha} \sum_{k=0}^{k_0} \frac{2^{-(1-\alpha)k}}{2 \cdot 2^{-k}|j_0|} \sum_{|i-j_0| \leq \lfloor 2^{-k}|j_0| \rfloor} |b(i)|
\]
\[
\leq 2^{2+\alpha} \sum_{k=0}^{k_0} 2^{-(1-\alpha)k} \frac{1}{2 \cdot \lfloor 2^{-k}|j_0| \rfloor + 1} \sum_{|i-j_0| \leq \lfloor 2^{-k}|j_0| \rfloor} |b(i)|,
\]
this last inequality follows from that $\lfloor 2^{-k} |j_0| \rfloor \leq 2^{-k} |j_0|$ and that
$\frac{2 \cdot \lfloor 2^{-k} |j_0| \rfloor + 1}{2 \cdot \lfloor 2^{-k} |j_0| \rfloor} \leq 2$ for each $k=0, ..., k_0$. Thus
\begin{equation} \label{s1}
\sum_{i \in I_1} \frac{|b(i)|}{|i-j_0|^{\alpha} |i+j_0|^{\beta}} \leq 2^{2+\alpha} \left( \sum_{k=0}^{+\infty} 2^{-(1-\alpha)k} \right)
(Mb)(j_0).
\end{equation}
Similarly, it is seen that
\begin{equation} \label{s2}
\sum_{i \in I_2} \frac{|b(i)|}{|i-j_0|^{\alpha} |i+j_0|^{\beta}} \leq C_{\beta} (Mb)(-j_0).
\end{equation}
Now, we estimate the last sum. If $i \in I_3$, then $|i \pm j_0| > \frac{|i|}{2}$, so
\begin{equation} \label{estim}
\sum_{i \in I_3} \frac{|b(i)|}{|i-j_0|^{\alpha} |i+j_0|^{\beta}} \leq C \sum_{|i|> 2 |j_0|} |i|^{-1} |b(i)| \leq 
C \| b \|_{\ell^{p}} |j_0|^{-1/p},
\end{equation}
the last inequality follows from the H\"older's inequality. Thus (\ref{estim}) implies that
\begin{equation} \label{s3}
\# \left\{ j \neq 0: \left| \sum_{i \in I_3} |i-j|^{-\alpha}|i+j|^{-\beta} b(i) \right| > \lambda \right\} \leq 
\left(C \frac{\|b\|_{\ell^{p}}}{\lambda} \right)^{p}, \,\, 1 \leq p < \infty.
\end{equation}
Finally, (\ref{s1}), (\ref{s2}), (\ref{s3}) and Theorem \ref{maximal} allow us to conclude that the operator $\widetilde{T}_{\alpha, \beta}$ is bounded from $\ell^{p}(\mathbb{Z})$ into $\ell^{p, \infty}(\mathbb{Z})$, for every $1 \leq p < \infty$. This completes the proof.
\end{proof}

\begin{remark} Let $0 \leq \gamma < 1$. Then the operator $T_{\alpha, \beta}$ is not bounded from $\ell^{p}(\mathbb{Z})$ into 
$\ell^{q}(\mathbb{Z})$ for $0 < p \leq \infty$ and $0 < q \leq \frac{1}{1 - \gamma}$. Indeed, by taking $b = \{ b(i) \}$ such that 
$b(0) =1$ and $b(i)=0$ for all $i \neq 0$ we have that $(T_{\gamma}b)(0) = 0$ and $(T_{\alpha, \beta} \, b)(j) = |j|^{\gamma - 1}$ 
for all $j \neq 0$. So, $b \in \ell^{p}(\mathbb{Z})$ but $T_{\alpha, \beta} \, b \notin \ell^{q}(\mathbb{Z})$ for $0 < p \leq \infty$ and 
$0 < q \leq \frac{1}{1 - \gamma}$. 

For $0 \leq \gamma < 1$, one also can see that $T_{\alpha, \beta}$ is not bounded from $\ell^{p}(\mathbb{Z})$ into 
$\ell^{q}(\mathbb{Z})$ for $\frac{1}{\gamma} \leq p \leq \infty$ and $0 < q \leq \infty$. For them, to consider $b = \{ b(i) \}$ with
$b(i)=0$ for $|i| \leq 1$, and $b(i) = \frac{1}{|i|^{\gamma} \log(|i|)}$ for $|i| \geq 2$. It is easy to check that
$b \in \ell^{p}(\mathbb{Z})$ for each $\frac{1}{\gamma} \leq p \leq \infty$, and $(T_{\alpha, \beta} \, b)(j) = +\infty$ for all $j$.

Let $0 \leq \gamma < 1$. If $\frac{1}{1-\gamma} < q < \infty$ and $0 < p \leq \frac{q}{1 + q \gamma}$, then the operator 
$T_{\alpha, \beta}$ is bounded from $\ell^{p}(\mathbb{Z})$ into $\ell^{q}(\mathbb{Z})$. This follows from Theorem \ref{lpq} and 
the embedding $\ell^{p_1}(\mathbb{Z}) \hookrightarrow \ell^{p_2}(\mathbb{Z})$ valid for $0 < p_1 < p_2 \leq \infty$.
\end{remark}

\begin{theorem} \label{Hpq} Let $0 \leq \gamma < 1$. If $T_{\alpha, \beta}$ is the operator given by $(\ref{Tgama})$, $0 < p \leq 1$ and 
$\frac{1}{q} = \frac{1}{p} - \gamma$, then
\[
\| T_{\alpha, \beta} \, b \|_{\ell^{q}(\mathbb{Z})} \leq C \| b \|_{H^{p}(\mathbb{Z})},
\]
where $C$ depends only on $\alpha$, $\beta$, $p$ and $q$.
\end{theorem}

\begin{proof} We take $p_0$ such that $1 < p_0 < \gamma^{-1}$. By Theorem \ref{decomp atomic}, given $b \in H^{p}(\mathbb{Z})$ we can write 
$b = \sum_k \lambda_k a_k$ where the $a_k$'s are $(p, \infty, d_p)$ atoms, the scalars $\lambda_k$ satisfies $\sum_{k} |\lambda_k |^{p} \leq C \| b \|_{H^{p}(\mathbb{Z})}^{p}$  and the series converges in $H^{p}(\mathbb{Z})$ and so in $\ell^{p_0}(\mathbb{Z})$ since
$H^{p}(\mathbb{Z}) \subset \ell^{p}(\mathbb{Z}) \subset \ell^{p_0}(\mathbb{Z})$ embed continuously. 
For $\frac{1}{q_0} = \frac{1}{p_0} - \gamma$, by Theorem \ref{lpq}, $T_{\alpha, \beta}$ is a bounded operator from 
$\ell^{p_0}(\mathbb{Z})$ into $\ell^{q_0}(\mathbb{Z})$. Since $b = \sum_k \lambda_k a_k$ in $\ell^{p_0}(\mathbb{Z})$, we have that 
$(T_{\alpha, \beta} \, b)(j) = \sum_{k} \lambda_k (T_{\alpha, \beta} \, a_k)(j)$ for all $j \in \mathbb{Z}$, and thus
\begin{equation}
|(T_{\alpha, \beta} \, b)(j)| \leq \sum_{k} |\lambda_k| |(T_{\alpha, \beta} \, a_k)(j)|, \,\,\,\, \forall \, j \in \mathbb{Z}. \label{puntual}
\end{equation}
If for $0 < p \leq 1$ and $\frac{1}{q}= \frac{1}{p} - \gamma$ we obtain that $\|T_{\alpha, \beta} \, a_k \|_{\ell^{q}(\mathbb{Z})} \leq C$, with $C$ independent of the $(p, \infty, d_p)$-atom $a_k$, then the estimate (\ref{puntual}) and the fact that 
$\sum_{k} |\lambda_k |^{p} \leq C \| b \|_{H^{p}(\mathbb{Z})}^{p}$ lead to 
\[
\|T_{\alpha, \beta} \, b \|_{\ell^{q}(\mathbb{Z})} \leq C \left( \sum_{k} |\lambda_k|^{\min\{1, q \}} \right)^{\frac{1}{\min\{1, q \}}} \leq C \left( \sum_{k} |\lambda_k |^{p} \right)^{1/p} \leq C\| b \|_{H^{p}(\mathbb{Z})}.
\] 
Since $b$ is an arbitrary element of $H^{p}(\mathbb{Z})$, the theorem follows.

To conclude the proof we will prove that for $0 < p \leq 1$ and $\frac{1}{q}= \frac{1}{p} - \gamma$ there exists an universal constant 
$C > 0$, which depends on $\alpha$, $\beta$, $p$ and $q$ only, such that 

\begin{equation} \label{uniform estimate}
\|T_{\alpha, \beta} \, a \|_{\ell^{q}(\mathbb{Z})} \leq C, \,\,\,\, \textit{for all} \,\, (p, \infty, d_p) - \textit{atom} \,\, a=\{ a(i) \}. 
\end{equation}
\\
To prove (\ref{uniform estimate}), let $J_{n_0}= \{ n_0 - m, ..., n_0, ..., n_0 + m \}$ be the support of the atom $a = \{a(i)\}$. We put
$3J_{\pm n_0} = \{ \pm n_0 - 3m, ..., \pm n_0, ..., \pm n_0 + 3m \}$. So

\begin{equation} \label{sum2}
\sum_{j \in \mathbb{Z}} |(T_{\alpha, \beta} \, a)(j)|^{q} = \sum_{j \in 3J_{n_0} \cup \, 3J_{-n_0}} |(T_{\alpha, \beta} \, a)(j)|^{q} + 
\sum_{j \in \mathbb{Z} \setminus (3J_{n_0} \cup \, 3J_{-n_0})} |(T_{\alpha, \beta} \, a)(j)|^{q}.
\end{equation}

To estimate the first sum, by taking into account that $\frac{q_0}{q} > 1$ we apply H\"older's inequality with that $\frac{q_0}{q}$, then from the $\ell^{p_0}(\mathbb{Z})$ - $\ell^{q_0}(\mathbb{Z})$ boundedness of $T_{\alpha, \beta}$, the size condition of the atom and since 
$\frac{1}{p} - \frac{1}{q} = \frac{1}{p_0} - \frac{1}{q_0} = \gamma$ we have
\begin{equation} \label{estim C}
\sum_{j \in 3J_{n_0} \cup \, 3J_{-n_0}}|(T_{\alpha, \beta} \, a)(j)|^{q} \leq 2 \cdot 3^{(q_0 - q)/q_0} \left( \sum_{j \in \mathbb{Z}}
|(T_{\alpha, \beta} \, a)(j)|^{q_0}\right)^{q/q_0} \cdot (\#J_{n_0})^{(q_0-q)/q_0}
\end{equation}
\[
\leq C \left( \sum_{j \in J_{n_0}}|a (j)|^{p_0}\right)^{q/p_0} \cdot (\#J_{n_0})^{(q_0-q)/q_0}
\]
\[
\leq C \, (\#J_{n_0})^{-q/p} \cdot (\#J_{n_0})^{q/p_0} \cdot (\#J_{n_0})^{(q_0-q)/q_0} = C,
\]
\\
with $C$ independent of $n_0$ and $m$.

To estimate the second sum in (\ref{sum2}), we denote $K(x,y) = |x - y|^{-\alpha} |x + y|^{-\beta},$ and we put 
$N - 1 = \lfloor p^{-1} - 1 \rfloor$. In view of the moment condition of $a$ we have, for 
$j \in \mathbb{Z} \setminus (3J_{n_0} \cup \, 3J_{-n_0})$, that

\[
(T_{\alpha, \beta} \, a)(j) = \sum_{i \in J_{n_0}} K(i,j) \, a(i) = \sum_{i \in J_{n_0}} [K(i,j) - q_{N}(i,j)] \, a(i),
\]
where $q_{N}(\, \cdot \,, j)$ is the degree $N - 1$ Taylor polynomial of the function $x \rightarrow K(x,j)$ expanded around $n_0$. 
By the standard estimate of the remainder term in the Taylor expansion there exists $\xi $ between $i$ and $n_0$ such that
\[
\left\vert K(i,j)-q_{N}\left( i,j\right) \right\vert \leq C \left\vert
i - n_0\right\vert ^{N} \left\vert \frac{\partial ^{N}}{\partial x^{N}}K(\xi, j)\right\vert
\]
\[
\leq C \left\vert i- n_0\right\vert ^{N} |j - \xi|^{-\alpha} |j + \xi|^{-\beta} \left( |j - \xi|^{-1} + |j + \xi|^{-1} \right)^{N},
\]
\\
where Proposition \ref{partialder} gives the last inequality. For $j \in \mathbb{Z} \setminus \left(3J_{n_0} \cup \, 3J_{-n_0}\right)$ we have $| j \pm n_0| \geq 2m$, since $\xi \in [n_0-m, n_0+m]$, it follows that $|\xi - n_0| \leq m  \leq \frac{1}{2} |j \pm n_0|$. So 
\[
|j \pm \xi|=|j \pm n_0 \mp n_0 \pm \xi| \geq |j \pm n_0| - | n_0 - \xi| \geq \frac{|j \pm n_0|}{2}
\] 
and then
\begin{equation} \label{estim K}
| K(i,j) - q_{N}(i, j) | \leq C \, | i- n_0 |^{N} |j - n_0|^{-\alpha} |j + n_0|^{-\beta} \left( |j - n_0|^{-1} + |j + n_0|^{-1} \right)^{N},
\end{equation}
\\
for every $j \in \mathbb{Z} \setminus \left(3J_{n_0} \cup \, 3J_{-n_0}\right)$ and $i \in J_{n_0}$. Now we decompose the set 
$R := \mathbb{Z} \setminus \left(3J_{n_0} \cup \, 3J_{-n_0}\right)$ by $R = R_{1} \cup R_{2}$ where
\[
R_{1} = \{j \in R : |n_0 - j| \leq |n_0 + j| \} \,\,\,\,\,\, \text{and} \,\,\,\,\,\, R_{2} = \{j \in R : |n_0 + j| < |n_0 - j| \}.
\]
\\
If $j \in R$, then $j \in R_{k}$ for some $k=1, 2$. From this, (\ref{estim K}) and since $\alpha + \beta = 1 - \gamma$, we obtain that

\[
| K(i,j) - q_{N}(i, j) | \leq C \, |i - n_0|^{N} |j + (-1)^{k} n_0|^{-(1-\gamma)-N}, 
\]
\\
for all $i \in J_{n_0}$ and all $j \in R_{k}$. This inequality gives
\begin{equation} \label{estim K2}
\sum_{j \in R} |(T_{\alpha, \beta} \, a)(j)|^{q} = \sum_{k=1}^{2} \sum_{j \in R_k} |(T_{\alpha, \beta} \, a)(j)|^{q}
\end{equation}
\[
=\sum_{k=1}^{2} \sum\limits_{j \in R_k}\left\vert \sum\limits_{i \in J_{n_0}}K(i,j) \, a(i)\right\vert^{q} = \sum_{k=1}^{2}\sum\limits_{j \in R_k}\left\vert \sum\limits_{i \in J_{n_0}} [K(i,j)- q_{N}(i,j)] \, a(i) \right\vert^{q}  
\]
\\
\[
\leq C  \left( \sum\limits_{i \in J_{n_0}}\left\vert i- n_0 \right\vert ^{N}\left\vert
a(i)\right\vert\right)^{q} \sum_{k=1}^{2}\sum\limits_{j \in \mathbb{Z} \setminus 3J_{(-1)^{k+1}n_0}}|j + (-1)^{k} n_0|^{-(1-\gamma)q-Nq}
\]

\[
\leq C \, m^{qN- \frac{q}{p}+q} \int_{m}^{\infty }t^{-q\left((1-\gamma)+N\right)}dt\leq C
\]
with $C$ independent of the $p-$atom $a$, since $-q\left((1-\gamma)+N\right)+1<0$. Finally, the inequality in (\ref{uniform estimate}) follows from (\ref{estim C}) and (\ref{estim K2}).
\end{proof}

\begin{remark}
Let $0 \leq \gamma < 1$. If $0 < q \leq \frac{1}{1-\gamma}$ and $0 < p \leq \frac{q}{1 + q \gamma}$, then the operator $T_{\alpha, \beta}$ is bounded from $H^{p}(\mathbb{Z})$ into $\ell^{q}(\mathbb{Z})$. This follows from Theorem \ref{Hpq} and the embedding 
$H^{p_1}(\mathbb{Z}) \hookrightarrow H^{p_2}(\mathbb{Z})$ valid for $0 < p_1 < p_2 \leq 1$.
\end{remark}

\begin{remark} \label{Igamma}
The argument utilized in the proof of Theorem \ref{Hpq} works for the discrete Riesz operator $I_{\gamma}$ as well. So, if $0 < \gamma < 1$,
$0 < q \leq \frac{1}{1-\gamma}$ and $0 < p \leq \frac{q}{1 + q \gamma}$, then the operator $I_{\gamma}$ is bounded from $H^{p}(\mathbb{Z})$ into $\ell^{q}(\mathbb{Z})$.
\end{remark}

\section{A Counter-example}

For $0 \leq \gamma < 1$, we consider the operator $U_{\gamma}$ given by
\[
(U_{\gamma}b)(j) = \sum_{i \neq \pm j} \frac{b(i)}{|i-j|^{\frac{1-\gamma}{2}} |i+j|^{\frac{1-\gamma}{2}}}.
\]
It is clear that $U_{\gamma} = T_{\frac{1-\gamma}{2}, \frac{1-\gamma}{2}}$. In this section, we will prove that there exists $\epsilon \in \left(0, \frac{1}{3} \right)$ such that, for every 
$0 \leq \gamma < \epsilon$, the operator $U_{\gamma}$ is not bounded from $H^{p}(\mathbb{Z})$ into $H^{q}(\mathbb{Z})$ 
for $0 < p \leq (1+\gamma)^{-1}$ and $\frac{1}{q} = \frac{1}{p} - \gamma$.

Let $b = \{ b(i) \}_{i \in \mathbb{Z}}$ be the sequence defined by $b(\pm 1) = 1$, $b(0)=-2$ and $b(i) = 0$ for all $i \neq -1, 0, 1$.
It is clear that $b$ satisfies the first moment condition, then $b \in H^{p}(\mathbb{Z})$ for every $\frac{1}{2} < p \leq 1$. In particular,
$b \in H^{(1+\gamma)^{-1}}(\mathbb{Z})$ for all $0 \leq \gamma < 1$. A computation gives
\[
(U_{\gamma}b)(-1) = b(0) = -2, \,\,\,\,\,\, (U_{\gamma}b)(0) = b(-1) + b(1) = 2, \,\,\,\,\,\, (U_{\gamma}b)(1) = b(0) = -2,
\]
\[
(U_{\gamma}b)(j) = \frac{2}{|j^{2}-1|^{(1-\gamma)/2}} - \frac{2}{|j|^{1-\gamma}}, \,\, \text{for all} \,\, j \neq 0, \pm 1.
\]
So
\[
\sum_{j=-\infty}^{+\infty} (U_{\gamma}b)(j) = -2 + 4 \sum_{j=2}^{+\infty} \left[ \frac{1}{(j^{2}-1)^{(1-\gamma)/2}} - \frac{1}{j^{1-\gamma}} \right].
\]

Next, we introduce two auxiliary functions. Let $g : [0,1) \to (0, +\infty)$ be the function defined by $g(\gamma) = \frac{1}{3^{\frac{1-\gamma}{2}}} - 
\frac{1}{2^{1-\gamma}}$, and let $h : \left[0, \frac{1}{3}\right) \to (0, +\infty)$ be given by 
$h(\gamma) = \frac{1}{2} - \frac{1}{8^{\frac{1-\gamma}{2}}}$. Since 
$g(0) = \frac{1}{\sqrt{3}} - \frac{1}{2} < \frac{1}{2} - \frac{1}{\sqrt{8}} = h(0)$, by the continuity of $g$ and $h$, there exists 
$\epsilon \in \left(0, \frac{1}{3} \right)$ such that $g(\gamma) < h(\gamma)$ for all $0 \leq \gamma < \epsilon$.

On the other hand, for $0 \leq \gamma < 1$, it is easy to check that 
$$(j+1)^{\gamma -1} < ((j+1)^{2}-1)^{(\gamma -1)/2} < j^{\gamma -1} < (j^{2}-1)^{(\gamma -1)/2}$$ for all $j \geq 2$ and so
\[
\bigcup_{j=3}^{+\infty} \left[ \frac{1}{j^{1-\gamma}}, \frac{1}{(j^{2}-1)^{(1-\gamma)/2}} \right] \subseteq 
\left(0, \frac{1}{8^{(1-\gamma)/2}} \right],
\]
being the previous union disjoint. We obtain
\[
\sum_{j=2}^{+\infty} \left[ \frac{1}{(j^{2}-1)^{(1-\gamma)/2}} - \frac{1}{j^{1-\gamma}} \right] = 
\sum_{j=3}^{+\infty} \left[ \frac{1}{(j^{2}-1)^{(1-\gamma)/2}} - \frac{1}{j^{1-\gamma}} \right] + g(\gamma)
\]
\[
\leq \frac{1}{8^{(1-\gamma)/2}} + g(\gamma) < \frac{1}{8^{(1-\gamma)/2}} + h(\gamma) = \frac{1}{2}, \,\,\, \text{for every} \,\,\, 
0 \leq \gamma < \epsilon,
\]
this implies
\[
\sum_{j=-\infty}^{+\infty} (U_{\gamma}b)(j) = -2 + 4 \sum_{j=2}^{+\infty} \left[ \frac{1}{(j^{2}-1)^{(1-\gamma)/2}} - \frac{1}{j^{1-\gamma}} \right] < 0, \,\,\, \text{for every} \,\,\, 0 \leq \gamma < \epsilon.
\]
Thus, by Corollary \ref{moment cond}, it follows that $U_{\gamma}$ is not bounded from $H^{(1+\gamma)^{-1}}(\mathbb{Z})$ into 
$H^{1}(\mathbb{Z})$ for every $0 \leq \gamma < \epsilon$.

For $0 < p < (1 + \gamma)^{-1}$, we take $L$ as any fixed integer with 
$L \geq \lfloor p^{-1}-1 \rfloor$, then by Corollary \ref{dense} the set $\mathcal{D}_{L}$ is dense in $H^{r}(\mathbb{Z})$ for each 
$p \leq r \leq 1$. In particular, there exists $c \in H^{p}(\mathbb{Z})$ such that $\| b - c \|_{H^{(1 + \gamma)^{-1}}(\mathbb{Z})} < \left| \sum_{j=-\infty}^{+\infty} (U_{\gamma} b)(j) \right|/2 \|U_{\gamma}\|_{H^{(1+\gamma)^{-1}} \to \ell^{1}}$. Then
\[
\left| \sum_{j=-\infty}^{+\infty} (U_{\gamma} c)(j)dx \right| \geq \left| \sum_{j=-\infty}^{+\infty} (U_{\gamma} b)(j) \right| - 
\sum_{j=-\infty}^{+\infty} \left| (U_{\gamma} (b-c))(j) \right|
\]
\[
\geq \left| \sum_{j=-\infty}^{+\infty} (U_{\gamma} b)(j) \right| -  \|U_{\gamma}\|_{H^{(1+\gamma)^{-1}} \to \ell^{1}} \| b - c \|_{H^{(1 + \gamma)^{-1}}(\mathbb{Z})} 
\]
\[
> \frac{1}{2} \left| \sum_{j=-\infty}^{+\infty} (U_{\gamma} b)(j) \right| > 0
\]
where the second inequality follows from Theorem \ref{Hpq} with $p = (1 + \gamma)^{-1}$ and $q=1$. But then, by Corollary \ref{moment cond}, for every $0 \leq \gamma < \epsilon$ the operator $U_{\gamma}$ is not bounded from $H^{p}(\mathbb{Z})$ into $H^{q}(\mathbb{Z})$ for each
$0 < p < (1 + \gamma)^{-1}$ and $\frac{1}{q}= \frac{1}{p} - \gamma$, since $\sum_{j=-\infty}^{+\infty} (U_{\gamma} c)(j) \neq 0$.

\

{\bf Acknowledgements.} I express my thanks to the referee for the useful comments and suggestions which helped me to improve the original manuscript.

\

\end{document}